\documentclass[10pt]{amsart}
\usepackage{amsmath}
\usepackage{amsxtra}
\usepackage{amscd}
\usepackage{amsthm}
\usepackage{amsfonts}
\usepackage{amssymb}
\usepackage{eucal}
\usepackage[all]{xy}
\usepackage{graphicx}
\usepackage{mathabx}
\usepackage[usenames]{color}

\newtheorem{theorem}{Theorem}[section]
\newtheorem{corollary}[theorem]{Corollary}
\newtheorem{lemma}[theorem]{Lemma}
\newtheorem{proposition}[theorem]{Proposition}

\newtheorem{thm-dfn}[theorem]{Theorem-Definition}
\newtheorem{claim}[theorem]{Claim}

\theoremstyle{definition}
\newtheorem{definition}[theorem]{Definition}
\newtheorem{remark}[theorem]{Remark}

\numberwithin{equation}{section}

\newcommand{\quash}[1]{}  

\newcommand{\bbD}{{\mathbb D}}

\newcommand{\bbG}{{\mathbb G}}

\newcommand{\bbQ}{{\mathbb Q}}

\newcommand{\bbZ}{{\mathbb Z}}

\newcommand{\calC}{{\mathcal C}}
\newcommand{\calD}{{\mathcal D}}
\newcommand{\calE}{{\mathcal E}}
\newcommand{\calF}{{\mathcal F}}
\newcommand{\calG}{{\mathcal G}}
\newcommand{\calH}{{\mathcal H}}

\newcommand{\calK}{{\mathcal K}}

\newcommand{\calO}{{\mathcal O}}

\newcommand{\Hom}{{\textnormal{Hom}}}
\newcommand{\op}{{\textnormal{op}}}
\newcommand{\wt}{\widetilde}
\newcommand{\Id}{\textnormal{Id}}
\newcommand{\DGj}{{\textnormal{DG}}} 

\newcommand{\CH}{{\textnormal{CH}(G)}}
\newcommand{\CHI}{{\textnormal{CH}(G_I)}}
\newcommand{\CHcdot}{{\textnormal{CH}(\cdot)}}
\newcommand{\DG}{{\DGj_{D(G \backslash G)}}}
\newcommand{\DGI}{{\DGj_{D(G_I \backslash G_I)}}}
\newcommand{\DGinv}{{\DGj^{-1}_{D(G \backslash G)}}}
\newcommand{\DGIinv}{{\DGj^{-1}_{D(G_I \backslash G_I)}}}

\title[On the Deligne-Lusztig involution for character sheaves]{On the Deligne-Lusztig involution for character sheaves}

\author{Alexander ~Yom Din}

\begin{document}
	
\date{\today}

\begin{abstract}
	For a reductive group $G$, we study the Drinfeld-Gaitsgory functor of the category of conjugation-equivariant $D$-modules on $G$. We show that this functor is an equivalence of categories, and that it has a filtration with layers expressed via parabolic induction of parabolic restriction. We use this to provide a conceptual definition of the Deligne-Lusztig involution on the set of isomorphism classes of irreducible character $D$-modules, which was defined previously in \cite[\S 15]{Lu1}.
\end{abstract}  

\maketitle

\tableofcontents


\section{Introduction}

\subsection{Background and motivation}

\subsubsection{The Deligne-Lusztig involution}

Let $G$ be a connected reductive algebraic group over a finite field $F$. On the set of isomorphism classes of irreducible representations of the finite group $G(F)$ (over $\bar{\bbQ}_{\ell}$) one has an involution $\textnormal{DL}$, the \emph{Deligne-Lusztig\footnote{Other names relevant for this involution are Alvis, Curtis and Kawanaka; See \cite[{[47]}]{Lu2} for the history of this involution.} involution}. Namely, in the $K$-group of representations, given an irreducible representation $V$, the irreducible representation $\textnormal{DL} (V)$ is given, up to a sign, by $$ \sum_{I \subset \Sigma} (-1)^{|I|} pind_{P^-_I}^G pres_{P_I}^G V.$$ Here $\Sigma$ is the set of simple roots, $P_I , P^-_I$ are opposite standard parabolics associated to the subset $I \subset \Sigma$, and $pres , pind$ denote parabolic restriction and induction.

\medskip

In \cite[\S 15]{Lu1}, an involution given by an analogous formula is defined on the set of isomorphism classes of irreducible character sheaves on $G$ (it is denoted ``$d$" there, we will denote it by $\textnormal{DL}$).

\medskip

In the present paper, we give a conceptual definition of this involution on irreducible character sheaves. Technically, we work with character $D$-modules rather than character sheaves, but one can transport everything to $\ell$-adic sheaves as well and, anyhow, in this introduction we will be vague about such details.

\medskip

What we actually do is to study the \emph{Drinfeld-Gaitsgory functor} of the $DG$-category of conjugation-equivariant sheaves on a connected reductive algebraic group $G$. The desired involution on irreducible character sheaves is then simply induced by this functor. The main technical result is a filtration of this functor whose layers are expressed as parabolic inductions of parabolic restrictions (with some cohomological shifts), in particular showing that our definition of the involution coincides with that of \cite[\S 15]{Lu1}.

\medskip

A point which seems nice for us to stress is that our definition of the Deligne-Lusztig involution for irreducible character sheaves is ``abstract" - its input is the category of conjugation-equivariant sheaves as a category, so it is not ``informed" about the more specific structure of Levi subgroups, parabolic induction functors and so on, as in the formula above (that is what we meant by the adjective ``conceputal" above).

\medskip

Let us next try to motivate the two main objects of this paper, character sheaves and the Drinfeld-Gaitsgory functor, independently of the utility of defining the Deligne-Lusztig involution.

\subsubsection{Character sheaves}

In the representation theory of finite groups, a prominent object is the space of \emph{class functions}. In somewhat fancy terms, one might say that the space of class functions is the cocenter of the category of (finite-dimensional) representations, and so one can assign to every representation an element of the space of class functions - its \emph{character}.

\medskip

Thus, after some familiarity with categorification, one might suspect that for an algebraic group $G$, the category of conjugation-equivariant sheaves on $G$ is of some similar basic importance. Restricting ourselves to the case when $G$ is a connected reductive group, it turns out that indeed a certain subcategory of this category (discovered by Lusztig), the subcategory of \emph{character sheaves}, is a central object of study. It should roughly be seen as a cocenter of some $2$-category of categorical representations of $G$. A fundamental prior role of character sheaves, studied in great depth by Lusztig, is their tight match with actual characters of irreducible representations of finite groups of Lie type under the sheaf-to-function dictionary.

\medskip

The whole category of conjugation-equivariant sheaves can be seen as some ``direct integral" of the subcategories of character sheaves with various ``central characters".

\subsubsection{The Drinfeld-Gaitsgory functor}

For any compactly-generated presentable $DG$-category $\calC$, one constructs an endo-functor $$\textnormal{DG}_{\calC} : \calC \to \calC,$$ the \emph{Drinfeld-Gaitsgory functor}. It seems to be a basic homological construction, describing some ``duality" phenomena. To get some feeling of that, consider the following examples:

\begin{itemize}
	\item In \cite{GaYo}, it is shown that under some conditions on the category $\calC$, the Drinfeld-Gaitsgory functor is an equivalence of categories, whose inverse is the Serre functor (which traditionally embodies some sort of ``duality"). However, these conditions (which can be thought of as ``smallness" conditions) oftentimes do not hold, and it is our feeling that in such cases the Drinfeld-Gaitsgory functor is more ``correct" than the Serre functor.
	\item The Drinfeld-Gaitsgory functor has a fondness to intertwine left and right adjoints - see claim \ref{clm DG commutes} for a general statement of this sort.
	\item For a ring $A$, the Drinfeld-Gaitsgory functor of the $DG$-category of $A$-modules is given by tensoring with the bimodule $$ \textnormal{Hom}_{A \otimes A^{\textnormal{op}}} (A , A \otimes A^{\textnormal{op}})$$ (where $A$ is considered as a bimodule over itself in the usual way, and the $\textnormal{Hom}$ is in the derived sense). Thus, the relation to Hochschild cohomology, Grothendieck-style duality, etc. is seen.
\end{itemize}

\medskip

See the introduction to \cite{GaYo} (as well as that paper itself) for some appearances of the Drinfeld-Gaitsgory functor (there called the \emph{Pseudo-Identity functor}) in representation theory.

\subsection{The results of this paper in short}

Let us summarize the main results of this paper in a more technical way. Let $G$ be a connected reductive algebraic group over an algebraically closed field $k$ of characteristic zero. Consider the category $D(G \backslash G)$ of $D$-modules on the quotient-stack of $G$ by the conjugation action of $G$, and consider the corresponding Drinfeld-Gaitsgory (a.k.a. Pseudo-Identity) functor $$\DG : D(G \backslash G) \to D(G \backslash G).$$ We prove:

\begin{itemize}
	
	\item (proposition \ref{clm DG commutes pind}) Given opposite parabolic subgroups $P, P^- \subset G$ with Levi $L$, there is a commutation relation $\DG \circ pind_P^G \cong pind_{P^-}^G \circ \DGj_{D(L \backslash L)}$, where $pind$ is the parabolic induction. In view of the second adjunction, this is shown quite formally, since $\DGj$ ``likes" to intertwine between left and right adjoints.
	
	\medskip
	
	\item (theorem \ref{thm main}) The functor $\DG$ is ``glued" in a specific way from cohomological shifts of functors of the form $pind^G_{P^-} \circ pres^G_P$, where $P,P^- \subset G$ are opposite parabolic subgroups and $pind , \ pres$ are functors of parabolic induction and restriction. This is done by ``resolving" the kernel governing $\DG$, by means of the wonderful compactification (technically, by means of the Vinberg semigroup).

	\medskip
	
	\item (proposition \ref{DG invertible}) The functor $\DG$ is invertible (i.e. an equivalence of categories). This follows quite formally from $\DG$ being proper (i.e. sending compact objects to compact objects), which in turn follows from the previous result.
	
	\medskip

	\item (proposition \ref{prp DG strat-exact}) Fixing an integer $d$, the functor $\DG [d]$ is $t$-exact when restricted to the subcategory of character $D$-modules which, roughly, are obtained via parabolic induction of cuspidal character $D$-modules from Levi's whose center has dimension $d$.
	
	\medskip
	
	\item (claim \ref{DL involutive}) The previous item provides an auto-bijection of the set of isomorphism classes of irreducible character $D$-modules; This auto-bijection is an involution. This is the \emph{Deligne-Lusztig involution}, appearing in \cite[\S 15]{Lu1} (where it is denoted by ``$d$").
	
	\medskip
	
	\item (proposition \ref{prp DL on Spr alpha}) Here, for illustration purposes, we recalculate a partial case of \cite[Corollary 15.8.(c)]{Lu1}. Namely, we calculate that when applied to the irreducible unipotent character $D$-modules of the ``principal series", i.e. constituents of the Springer $D$-module, which are parametrized by irreducible representations $\alpha$ of the Weyl group $W$, the Deligne-Lusztig involution swaps $\alpha$ with $\textnormal{sgn} \otimes \alpha$, where $\textnormal{sgn}$ is the one-dimensional sign representation of $W$. This is calculated using the curious formula $$ \sum_{I \subset \Sigma} (-1)^{|I|} \cdot ind^W_{W_I} res^W_{W_I} V = \textnormal{sgn} \otimes V$$ in the $K$-group of finite-dimensional representations of the Weyl group $W$ (here $\Sigma$ is the set of simple reflections and $W_I \subset W$ are the various ``parabolic" subgroups), for which references are given.
\end{itemize}

\subsection{Acknowledgments}

I would like to thank the anonymous referee for detecting some errors in section \ref{sec DL on prin ser}. I would like to thank D. Gaitsgory for detecting errors in section \ref{sec DG functor}. I would like to thank Xinwen Zhu and Nir Avni for helpful discussions. I would like to thank Roman Bezrukavnikov, Sam Gunningham and George Lusztig for helpful correspondence.

\section{Notations and conventions}

\subsection{Categories}\label{subsec bg cat}

We fix an algebraically closed field $k$ of characteristic zero. We will work with the $(\infty , 1)$-category $Lin_k$ of $k$-linear, stable and presentable $(\infty,1)$-categories, where morphisms are continuous $k$-linear functors (i.e. $k$-linear functors preserving colimits). By a \emph{category} we will mean an object in $Lin_k$ unless remarked otherwise, and by a \emph{functor} we will mean a morphism in $Lin_k$, unless remarked otherwise.

\medskip

By a \emph{subcategory} of a category we will mean a full subcategory closed under colimits which is presentable (and thus itself an object of $Lin_k$).

\medskip

For a functor $F$, we denote by $F^R$ and $F^L$ the right and left adjoints of $F$. We say that a continuous functor between compactly generated categories is \emph{proper}, if it sends compact objects to compact objects. This is equivalent to the functor admitting a continuous right adjoint.

\medskip

Given a $t$-structure on $\calC \in Lin_k$, we will say that an object $\calF \in \calC$ is \emph{irreducible} if $\calF$ lies in $\calC^{\heartsuit}$ and is irreducible as an object of this abelian category. We will say that an object $\calF \in \calC$ is \emph{bounded} if $H^n (\calF) \neq 0$ only for finitely many $n \in \bbZ$. We will say that an object $\calF \in \calC$ has \emph{finite length} if it is bounded an each of its cohomologies has finite length in the abelian category $\calC^{\heartsuit}$ (denote by $\calC^{fl} \subset \calC$ the non-cocomplete subcategory of objects of finite length). We will say that an object $\calF \in \calC$ is \textnormal{semisimple} if it is isomorphic to a finite direct sum of cohomological shifts of semisimple objects of finite length in $\calC^{\heartsuit}$ (denote by $\calC^{ss} \subset \calC$ the non-cocomplete subcategory of semisimple objects).

\medskip

\begin{definition}\label{def glue} Let $\calC \in Lin_k$. We say that an object $x \in \calC$ is \emph{glued} from a list of objects $x_0 , \ldots , x_n \in \calC$ if there exist fiber sequences $$ y_{n-1} \to x \to x_n,$$ $$ y_{n-2} \to y_{n-1} \to x_{n-1},$$ $$ \ldots $$ $$ x_0 \to y_{1} \to x_{1}.$$
\end{definition}

\quash{\begin{definition} Let $\calC \in Lin_k$.
	\begin{enumerate}
		\item Let us say that an object $x \in \calC$ is \emph{glued} from a list of objects $x_0 , \ldots , x_n \in \calC$ if there exist fiber sequences $$ y_{n-1} \to x \to x_n,$$ $$ y_{n-2} \to y_{n-1} \to x_{n-1},$$ $$ \ldots $$ $$ x_0 \to y_{1} \to x_{1}.$$
		\item Let us say that $\calC$ is \emph{glued} from a list of subcategories $\calC_0 , \ldots , \calC_n \subset \calC$ if for every object $x \in \calC$, there exists a list of objects $x_0 , \ldots , x_n \in \calC$ with $x_i \in \calC_i$, such that $x$ is glued from the list $x_0 , \ldots , x_n$.
		\item Let us say that $\calC$ is \emph{glued semi-orthogonally} from a list of subcategories $\calC_0 , \ldots , \calC_n \subset \calC$ if it is glued from it, and in addition $\textnormal{Hom} (x_j , x_i) = 0$ whenever $x_j \in \calC_j , x_i \in \calC_i$ and $j < i$.
	\end{enumerate}
\end{definition}

\begin{lemma}
	Let $\calC \in Lin_k$. Suppose that $\calC$ is glued semi-orthogonally from a list of subcategories $\calC_0 , \ldots , \calC_n$. Suppose in addition that $\calC$ is compactly generated.
\end{lemma}

\begin{lemma}
	Let $\calC \in Lin_k$. Suppose that $\calC$ is glued semi-orthogonally from a list of subcategories $\calC_0 , \ldots , \calC_n$. Suppose in addition that $\textnormal{Hom} (x_i , x_j) = 0$ whenever $x_j \in \calC_j , x_i \in \calC_i$ and $j < i$. Then $$ \calC = \bigoplus_{0 \leq i \leq n} \calC_i.$$
\end{lemma}}

\subsection{Spaces}

By a \emph{scheme} we will mean a scheme of finite type over $k$. For convenience, by a \emph{space} we will mean a QCA stack of finite type over $k$ (see \cite[Definition 1.1.8]{DrGa1} for the notion of a QCA stack).

\medskip

For a space $X$, we will denote by $D(X)$ the category of $D$-modules on $X$. Recall (see \cite{DrGa1}) that $D(X)$ is compactly generated and self-dual (by Verdier duality). We denote by $\omega \in D(X)$ the ``dualizing sheaf" ($\omega = \pi^! k$ where $\pi :  X \to \bullet$), by $C \in D(X)$ the ``constant sheaf" ($C = \bbD^{\textnormal{Ve}} (\omega) = \pi^* k$) and, if $X$ is smooth, by $\calO$ the constant $D$-module ($\calO = \pi^{\circ} k = \pi^! k [n] = \pi^* k [-n]$ where $n$ is the dimension of $X$).

\subsection{The kernel formalism}

Given spaces $X$ and $Y$, we identify $$D(X \times Y) \approx \textnormal{Hom}_{Lin_k} (D(X) , D(Y))$$ by matching a ``kernel" $\calK \in D(X \times Y)$ with the functor $T_{\calK} (\calF) = (\pi_2)_{\blacktriangle} (\calK \overset{!}{\otimes} \pi_1^! \calF)$, where $(\pi_2)_{\blacktriangle}$ is the \emph{renormalized direct image} (see \cite[Definition 9.3.2]{DrGa1}).

\subsection{The group}

We fix a connected reductive group $G$. \quash{For a parabolic subgroup $P \subset G$, we denote by $U$ the unipotent radical of $P$ and by $M$ the Levi quotient of $P$. Given a parabolic subgroup $P^- \subset G$ opposite to $P$, we consider $M = P \cap P^-$ both a quotient of $P$ and a quotient of $P^-$.}

\medskip

We denote by $T$ the universal Cartan of $G$. We denote by $\Sigma \subset X^* (T)$ the set of simple roots, and by $W \subset Aut(T)$ the Weyl group.

\medskip

We fix a Torel $T_{sub} \subset B \subset G$, i.e. a Cartan subgroup $T_{sub}$ contained in a Borel subgroup $B$. We then have an identification $\phi : T \xrightarrow{\sim} T_{sub}$. For every $I \subset \Sigma$, we denote by $P_I \subset G$ the corresponding standard parabolic subgroup containing $B$ (with the convention $P_{\emptyset} = B$ and $P_{\Sigma} = G$), and by $P^-_I$ we denote the corresponding opposite parabolic containing $T_{sub}$. We denote $G_I := P_I \cap P_I^-$, and mostly think of it in the usual way as a quotient of $P_I$ and a quotient of $P_I^-$.

\medskip

For $0 \leq i \leq |\Sigma|$ we denote $d_i := \dim T - i = \dim Z(G) + |\Sigma| - i $ and for $I \subset \Sigma$ we denote $d_I := \dim Z(G_I) = d_{|I|}$.

\section{The Drinfeld-Gaitsgory functor}\label{sec DG functor}

In this section we recall the Drinfeld-Gaitsgory functor and prove some properties of it (which mostly can be extracted from \cite{Ga1}, but given here in greater generality).

\subsection{The definition}

Let us recall the definition of the \emph{Drinfeld-Gaitsgory functor} $\DGj_{\calC} : \calC \to \calC$, where $\calC \in Lin_k$ is compactly generated (see \cite{Ga1} and also \cite[Section 1.4.2]{GaYo}, where it is denoted $\textnormal{Ps-Id}_{\calC}$).

\medskip

For a compactly generated $\calC \in Lin_k$, we have a colimit-preserving functor\footnote{Here, the target is not an object of $Lin_k$ - i.e. we step briefly outside of our ``world".} $$ (\cdot)^{\vee} : \calC \to (\calC^{\vee})^{\op}$$ characterized by $$ \Hom (c^{\op} , m^{\vee}) \cong \Hom (m , c)$$ for $m \in \calC , c \in \calC^c$ (and $c^{\op} \in (\calC^c)^{\op}$ denotes the corresponding object in the opposite category). An object $m \in \calC$ is called \emph{reflexive}, if the natural map $m \to (m^{\vee})^{\vee}$ is an isomorphism. Compact objects are reflextive, and \cite[Corollary 6.1.8]{Ga1} shows that coherent objects in $D(X)$, for a space $X$, are reflexive as well.

\medskip

It will be convenient in what follows to keep in mind the identification $$ \Hom (\calC , \calD) \cong \calC^{\vee} \otimes \calD.$$ Given two compactly generated categories $\calC , \calD \in Lin_k$ we denote $$ \wt{\DGj}_{\calC , \calD} : \Hom (\calC , \calD) \cong \calC^{\vee} \otimes \calD \xrightarrow{(\cdot)^{\vee}} (\calD^{\vee} \otimes \calC)^{\op} \cong \Hom (\calD , \calC)^{\op}.$$ \quash{by sending $(c , d) \in (\calC^c)^{\op} \times \calD^c$ to $(d,c) \in (\calD^c)^{\op} \times \calC$.} Then the Drinfeld-Gaitsgory functor $\DGj_{\calC} : \calC \to \calC$ is given as $$\DGj_{\calC} := \wt{\DGj}_{\calC , \calC} (\textnormal{Id}_{\calC}).$$

\medskip

Let us recall (see \cite[Lemma 2.1.3]{GaYo}) that for a space $X$, the functor $\DGj_{D(X)}$ is given by the kernel $\Delta_! C$ where $\Delta : X \to X \times X$ is the diagonal.

\subsection{Intertwining left and right adjoints}

\begin{claim}\label{clm DG commutes}
	Let $\calC , \calD \in Lin_k$ be compactly generated categories, and let $F : \calC \to \calD$ be such that $F^R$ and $F^{RR}$ exist and are continuous, as well as $F^L$ and $F^{LL}$ exist (in other words, both $F$ and $F^{\vee}$ should admit a proper right adjoint). Then the following diagram commutes: $$\xymatrix{ \calD \ar[r]^{\DGj_{\calD}} \ar[d]_{F^R} & \calD \ar[d]^{F^L} \\ \calC \ar[r]_{\DGj_{\calC}} & \calC}.$$
\end{claim}

\begin{proof}
	Recall the for a continuous functor admitting a left adjoint, the conjugate of the the left adjoint is the same as the dual (see \cite{Ga1} for all these terms). We thus see that we have a commutative diagram \begin{equation}\label{eq DG abstract} \xymatrix{\calD^{\vee} \otimes \calD \ar[d]_{F^{\vee} \otimes -} \ar[r]^{\wt{\DGj}_{\calD , \calD}} & (\calD^{\vee} \otimes \calD)^{\op} \ar[d] \ar[d]^{- \otimes F^L} \\ \calC^{\vee} \otimes \calD \ar[r]^{\wt{\DGj}_{\calC , \calD}} & (\calD^{\vee} \otimes \calC)^{\op} \\ \calC^{\vee} \otimes \calC \ar[u]^{- \otimes F} \ar[r]^{\wt{\DGj}_{\calC , \calC}} & (\calC^{\vee} \otimes \calC)^{\op} \ar[u]_{(F^R)^{\vee} \otimes -} }. \end{equation} Then, $\Id_{\calC}$ in the bottom left and $\Id_{\calD}$ in the top left are mapped to the same thing in the middle right as is evident by going vertically and then horizontally. Going now horizontally and then vertically, we get the desired identity.

\end{proof}

\subsection{Invertibility}

\begin{claim}\label{clm DG inv if prp}
	Let $\calC \in Lin_k$ be a compactly generated category. Suppose that $\Id_{\calC} \in \Hom (\calC , \calC)$ is reflexive. If $\DGj_{\calC}$ admits a proper right adjoint then it is right-invertible. In particular, if $\DGj_{\calC}$ and $\DGj_{\calC^{\vee}}$ admit proper right adjoints, then $\DGj_{\calC}$ is invertible.
\end{claim}

\begin{proof}
	Notice that the condition that $\Id_{\calC} \in \Hom (\calC , \calC)$ is reflexive means $$\wt{\DGj}_{\calC , \calC} (\DGj_{\calC}) \cong \Id_{\calC}.$$ Considering again the lower square of diagram \ref{eq DG abstract}, with $\calD := \calC$ and $F := \DGj_{\calC}$, evaluating on the object $\Id_{\calC}$ in the left bottom, we obtain $$ \Id_{\calC} \cong \DGj_{\calC} \circ (\DGj_{\calC})^R.$$ Thus, $\DGj_{\calC}$ is right-invertible.

\medskip

	The second assertion follows by recalling that $\DGj_{\calC^{\vee}} \cong (\DGj_{\calC})^{\vee}$, so that if $\DGj_{\calC^{\vee}}$ admits a proper right adjoint then, by what we just proved, $(\DGj_{\calC})^{\vee}$ is right-invertible, and thus $\DGj_{\calC}$ is left-invertible.
\end{proof}

\begin{corollary}[{\cite[Corollary 6.7.2]{Ga1}}]\label{cor DG inv if prp}
	Let $X$ be a space. If $\DGj_{D(X)}$ admits a proper right adjoint, then $\DGj_{D(X)}$ is invertible.
\end{corollary}

\begin{proof}
	First, notice that $\Id_{D(X)} \in \Hom (D(X) , D(X))$ is reflexive, since it is given by a coherent kernel (as is easily seen from preservation of holonomicity by functors!), and as recalled above, coherent objects are reflexive. Second, recall that $D(X)^{\vee} \cong D(X)$ via Verdier duality. Thus the corollary follows from the claim.
\end{proof}

\section{Adjoint-equivariant $D$-modules}

In this section we gather information regarding conjugation-equivariant $D$-modules, their parabolic restriction and induction, and character $D$-modules.

\subsection{Parabolic induction and restriction}

Let $P \subset G$ be a parabolic, with Levi quotient $M$. The functor of \emph{parabolic restriction} $$pres^G_P : D( G \backslash G) \to D(M \backslash M)$$ is defined as $q_* p^!$ where

\begin{equation}\label{eq basic cor}
	\xymatrix{& P \ \backslash \ P \ar[ld]_p \ar[rd]^q & \\ G \ \backslash \ G & & M \ \backslash \ M}
\end{equation} is the natural correspondence.

The map $p$ is projective and the map $q$ is smooth of relative dimension $0$ (notice that $q$ is also safe!). Hence $pres^G_P$ admits a left adjoint, the functor of \emph{parabolic induction} $$pind^G_P : D(M \backslash M) \to D(G \backslash G)$$ given by $p_! q^* \cong p_* q^!$.

\medskip

For $I \subset \Sigma$, we abbreviate: $$pres_I := pres^G_{P_I}, \quad pres_I^- := pres^G_{P^-_I} , \quad pind_I := pind^G_{P_I} , \quad pind_I^- := pind^G_{P_I^-}.$$

\begin{remark}\label{rem properties of pind}
For $J \subset I \subset \Sigma$, let us denote by $pind_J^I : D(G_J \backslash G_J) \to D(G_I \backslash G_I)$ the analogous functor for $G_I$. One has $pind_I \circ pind_J^I \cong pind_J$. Also, ``Mackey theory" shows that for $I,J \subset \Sigma$, the functor $pres_J \circ pind_I$ can be glued from functors of the form $ pind_{K_2}^J \circ ? \circ pres_{K_1}^I$ where $K_1 \subset I$ and $K_2 \subset J$ with $|K_1| = |K_2|$ and $?$ is an equivalence given by conjugation by a suitable element in $G$. For example, in the sepcial case $I = \emptyset$, the functor $pres_J \circ pind_{\emptyset}$ can be glued from functors of the form $pind_{\emptyset}^J \circ \dot{w}^!$ where $\dot{w} : T \backslash T \to T \backslash T$ is the conjugation by $\dot{w} \in N_G (T)$. See \cite{Gu1} for explicit details in the case $G \backslash \mathfrak{g}$ (everything transfers to the case $G \backslash G$ word-by-word).
\end{remark}

\subsection{Second adjointness and exactness}

The proof of the following theorem using Braden's hyperbolic localization theorem is sketched in \cite[Section 0.2.1]{DrGa3}:

\begin{theorem}[Second adjointness]\label{thm second adj}
	The functor $pres_I$ is left adjoint to $pind_I^-$.
\end{theorem}

\begin{corollary}\label{cor second adj}
	The functors $pres_I$ and $pind_I$ admit all iterated left and right adjoints (in particular, $pres_I$ is proper).
\end{corollary}

The following theorem is proved in \cite{BeYo}:

\begin{theorem}[{\cite[Theorem 5.4]{BeYo}}]\label{thm exactness}
	The functors $pres_I$ and $pind_I$ are $t$-exact.
\end{theorem}

\subsection{Character $D$-modules}

The functor $$ ch : D(G \backslash (G/U \times G/U) / T) \to D(G \backslash G)$$ is defined (perhaps up to a cohomological shift, which is irrelevant to us) as $p_* q^!$ where

\begin{equation}\label{eq basic cor 2}
\xymatrix{& G \ \backslash \ (G \times G/B) \ar[ld]_p \ar[rd]^q & \\ G \ \backslash \ G & & G \ \backslash \ (G/U \times G/U) / T}
\end{equation} (here $B \subset G$ is a Borel subgroup with unipotent radical $U$ and Levi quotient $T$, and the maps are $p(g , xB) = g , q(g , xB) = (xU , gxU)$). Notice that $p$ is proper and $q$ is smooth, hence $ch$ is proper.

\medskip

\begin{remark}\label{rem properties of ch}
	Denoting by $ch^I$ the analogous functor for $G_I$, similarly to remark \ref{rem properties of pind} one shows that $ pind_I \circ ch^I \cong ch \circ ?$ for some functor $?$, and also that $pres_I \circ ch$ can be glued from functors of the type $ch^I \circ ?$ where $?$ is some functor.
\end{remark}

\quash{One verifies that \begin{equation}\label{eq pind and ch} pind_I \circ ch^I\cong ch \circ r_* s^! \end{equation} where $$ \xymatrix{& P_I \ \backslash \ (P_I / B \times P_I / B)\ar[ld]_r \ar[rd]^s & \\ G \ \backslash \ (X \times X) & & G_I \ \backslash \ (X_I \times X_I)}. $$}

\begin{definition}
	The subcategory $\CH \subset D(G \backslash G)$ of \emph{character $D$-modules} is the subcategory generated under colimits by the image of $ch$ on $T$-monodromic objects.
\end{definition}

The following are standard properties:

\begin{lemma}\label{lem irr ch}\
	\begin{enumerate}
		\item The irreducible subquotients of cohomologies of any character $D$-module are again character $D$-modules.
		\item Every irreducible character $D$-module is of geometric origin (in particular, holonomic with regular singularities).
		\item Every compact character $D$-module has finite length.
		\item $\CH \subset D(G \backslash G)$ is closed under truncation.
	\end{enumerate}
\end{lemma}

\begin{lemma}\label{lem pres pind preserve ss}
	The functors $pind_I$ and $pres_I$ preserve $\CHcdot$.
\end{lemma}

\begin{proof}
	This follows from remark \ref{rem properties of ch}.
\end{proof}

\begin{lemma}\label{lem pres pind preserve ss}
	The functors $pind_I$ and $pres_I$ preserve $\CHcdot^{ss}$ (and hence also $\CHcdot^{fl}$).
\end{lemma}

\begin{proof}
	
	In view of lemma \ref{lem irr ch}, every irreducible character $D$-module is of geometric origin. Hence by the decomposition theorem, $pind_I$ sends irreducible character $D$-modules to semisimple ones.
	
	\medskip
	
	To show that $pres_I$ preserves semisimplicity of character $D$-modules, one uses its preservation of purity - see \cite[\S 5.3]{BeYo}.
	
\end{proof}

\begin{lemma}\label{lem pind pind- same}
	The functors $pind_I$ and $pind_I^-$ induce the same map $$K_0 (\CHI^{fl}) \to K_0 (\CH^{fl}).$$
\end{lemma}

\begin{proof}
	
	Let $\calG \in \CHI^{\heartsuit}$ be irreducible; We want to show that $pind_I (\calG)$ and $pind^-_I (\calG)$ are equal in the $K_0$-group. Since these objects are in the heart (by theorem \ref{thm exactness}) and semisimple (by lemma \ref{lem pres pind preserve ss}), it is enough to show that for every irreducible $\calF \in \CH^{\heartsuit}$ we have $$ [pind_I (\calG) : \calF] = [pind^-_I (\calG) : \calF]$$ (where $[ - : \calF ]$ denotes the amount of times $\calF$ enters the semisimple $-$). And indeed: $$ [pind_I (\calG) : \calF] = \dim H^0 \textnormal{Hom} (pind_I (\calG) , \calF) = \dim H^0 \textnormal{Hom} (\calG , pres_I (\calF)) = [pres_I (\calF) : \calG] = $$ $$  = \dim H^0 \textnormal{Hom} (pres_I (\calF) , \calG) = \dim H^0 \textnormal{Hom} (\calF , pind_I^- (\calG)) = [ pind_I^- (\calG) : \calF]$$ (where we have also used $pres_I (\calF)$ being in the heart (by theorem \ref{thm exactness}) and semisimple (by lemma \ref{lem pres pind preserve ss})).
	
\end{proof}

\subsection{Decomposition w.r.t. cuspidal rank}

For $0 \leq i \leq |\Sigma | $, we denote by $\CH^{(\leq i)}  \subset \CH$ the subcategory generated under colimits by the images of the functors $pind_I$, where $|I| \leq i$. We also denote by $\CH^{(i)}$ the right-orthogonal of $\CH^{(\leq i-1)}$ in $\CH^{(\leq i)}$ (since the $pind_I$'s are proper, these again are subcategories in the sense of subsection \ref{subsec bg cat}). In particular, we set $\CH^{cusp} := \CH^{(|\Sigma |)}$ (the subcategory of \emph{cuspidal} objects).

\medskip

\begin{lemma}\label{lem irr uch cusp rank}
	Let $\calF \in \CH^{\heartsuit}$ be irreducible. Then there exists $I \subset \Sigma$ and a cuspidal irreducible $\calG \in \CHI^{\heartsuit}$ such that $\calF$ is isomorphic to a direct summand of $pind_I \calG$. One has then $\calF \in \CH^{(|I|)}$.
\end{lemma}

\begin{proof}
		
	Consider a minimal $I$ for which there exists irreducible $\calG \in \CHI^{\heartsuit}$ such that $\calF$ is a direct summand of $pind_I \calG$ (such $I$ exists because $\Sigma$ always suits). We want to show that $\calG$ is cuspidal. Otherwise, we would have $J \subsetneq I$ such that $pres^I_J \calG \neq 0$. Taking an irreducible quotient $pres^I_J \calG \to \calH$, by adjunction we get a non-zero map $\calG \to pind^I_J \calH$. By semisimplicity, $\calG$ is a direct summand of $pind^I_J \calH$ and hence $pind_I \calG$ is a direct summand of $pind_I pind^I_J \calH \cong pind_J \calH$, so $\calF$ is a direct summand of $pind_I \calH$, contradicting the minimality of $I$.
	
	\medskip
	
	Let $\calF$ be as above. Clearly $\calF \in \CH^{(\leq |I|)}$. Moreover, from remark \ref{rem properties of pind} we see that $\calF$ is in the right-orthogonal to $\CH^{(\leq |I|-1)}$.

\end{proof}

\begin{definition}
	Let $\calF \in \CH^{\heartsuit}$ be irreducible. We define the \emph{cuspidal rank} of $\calF$ as the integer $0 \leq i \leq |\Sigma|$ for which $\calF \in \CH^{(i)}$.
\end{definition}

\begin{lemma}\label{lem no Hom between different cusp rank}
	Let $\calF_1 , \calF_2 \in \CH^{\heartsuit}$ be irreducibles of cuspidal ranks $i_1 , i_2$. If $i_1 \neq i_2$, then $\textnormal{Hom} (\calF_1 , \calF_2) = 0$.
\end{lemma}

\begin{proof}
	If $i_1 < i_2$, the claim is immediate. Suppose that $i_1 > i_2$. By lemma \ref{lem irr uch cusp rank}, we can find $|I| = i_2$ and irreducible cuspidal $\calG \in \CHI^{\heartsuit}$ such that $\calF_2$ is a direct summand of $pind_I (\calG)$. Thus it is enough to show that $\textnormal{Hom} (\calF_1 , pind_I (\calG) ) = 0$. By second adjunction, this $\textnormal{Hom}$ is the same as $\textnormal{Hom} (pres^-_I \calF_1 , \calG)$, and hence it is zero since $pres^-_I \calF_1 = 0$.
\end{proof}

Let us denote by $$ inc^{(\leq i-1)} : \CH^{(\leq i - 1)} \to \CH^{(\leq i)} , \quad inc^{(i)} : \CH^{(i)} \to \CH^{(\leq i)} $$ the inclusion functors. Let us denote by $$ P^{(\leq i-1)} = inc^{(\leq i-1)} \circ (inc^{(\leq i-1)})^R, \quad P^{(i)} = inc^{(i)} \circ (inc^{(i)})^L$$ the corresponding ``projection" endo-functors of $\CH^{(\leq i)}$. One has a fiber sequence $$ P^{(\leq i-1)} \to \textnormal{Id} \to P^{(i)}.$$

\begin{lemma}\label{lem proj fiber seq splits}
	Let $\calF \in \CH^{(\leq i)}$ be of finite length. Then the fiber sequence $$ P^{(\leq i-1)}(\calF) \to \calF \to P^{(i)} (\calF)$$ splits.
\end{lemma}

\begin{proof}
	
	\emph{Step 1}: Let us show first that $P^{(i)} (\calF)$ (resp. $P^{(\leq i)}(\calF)$) is of finite length, with all irreducible constituents being cuspidal of rank $i$ (resp. $\leq i-1$). We reduce to $\calF$ being irreducible. Then if $\calF$ has cuspidal rank $i$, $\calF \to P^{(i)}(\calF)$ is an isomorphism. If $\calF$ has cuspidal rank $\leq i-1$, $P^{(\leq i-1)}(\calF) \to \calF$ is an isomorphism.
	
	\medskip
	
	\emph{Step 2}: The fiber sequence splits, since by the first step and by lemma \ref{lem no Hom between different cusp rank}, we have $$\textnormal{Hom} (P^{(i)}(\calF) , P^{(\leq i-1)} (\calF)) = 0.$$
\end{proof}

\begin{corollary}\label{cor inclusion proper}
	The inclusion $\CH^{(i)} \to \CH^{(\leq i)}$ is proper.
\end{corollary}

\begin{proof}
	The objects $ (inc^{(i)})^L (\calF)$, where $\calF \in \CH^{(\leq i)}$ are compact, are compact generators of $\CH^{(i)}$. Hence it is enough to show that, for compact $\calF \in \CH^{(\leq i)}$, the object $$ inc^{(i)} ((inc^{(i)})^L(\calF)) = P^{(i)}(\calF) \in \CH^{(\leq i)}$$ is compact. This follows from lemma \ref{lem proj fiber seq splits} (recall that compact objects in $\CH$ have finite length).
\end{proof}

\begin{proposition}\label{clm left and right orth}
	$\CH^{(i)}$ is the left-orthogonal of $\CH^{(\leq i-1)}$ in $\CH^{(\leq i)}$.
\end{proposition}

\begin{proof}
	Let $\calF \in \CH^{(i)}$ and $\calG \in \CH^{(\leq i-1)}$; We want to show that $\textnormal{Hom} (\calF , \calG) = 0$. One can assume that $\calF$ is compact in $\CH^{(i)}$, and hence in $\CH^{(\leq i)}$ by corollary \ref{cor inclusion proper}. This allows to assume that $\calG$ is compact. One then reduces to $\calG$ being of the form $pind_I (\calH)$ where $|I| \leq i - 1$ and $\calH$ is compact. By second adjointness one has $$ \textnormal{Hom} (\calF , pind_I (\calH)) \cong \textnormal{Hom} (pres_I^- (\calF) , \calH).$$ But $pres_I^- (\calF) = 0$, and we are done.
\end{proof}

\medskip

\begin{theorem}\label{thm cusp decomp}
	One has a direct sum decomposition $$ \CH = \bigoplus_{0 \leq i \leq |\Sigma |} \CH^{(i)}$$ compatible with the $t$-structure.
\end{theorem}

\begin{proof} One splits the filtration $$ \CH^{(\leq 0)} \subset \ldots \subset \CH^{(\leq |\Sigma |)}$$ using proposition \ref{clm left and right orth}. The compatibility with the $t$-structure follows from the $t$-exactness of the $pres_I$'s.
\end{proof}

\begin{remark}
	See \cite{Gu2} for a statement, in the case of $G \backslash \mathfrak{g}$, which both generalizes to all $D$-modules (rather than character $D$-modules) and also takes into account the more refined ``cuspidal support" (versus only ``cuspidal rank").
\end{remark}

\section{The main results}

In this section we state the main results of this paper.

\subsection{Commutation with parabolic induction}

\begin{proposition}\label{clm DG commutes pind}
	The following diagram is commutative:

	$$ \xymatrix@R+1pc@C+2pc{ D(G \backslash G ) \ar[r]^{\DG} & D(G \backslash G) \\ D(G_I \backslash G_I) \ar[r]_{\DGI} \ar[u]^{pind_I^-} & D(G_I \backslash G_I) \ar[u]_{pind_I}} $$
\end{proposition}

\begin{proof}
	Apply claim \ref{clm DG commutes} to $F = pres_I$ (taking into account theorem \ref{thm second adj} and corollary \ref{cor second adj}).
\end{proof}

\subsection{Filtration}

\begin{theorem}\label{thm main}
	For $0 \leq i \leq |\Sigma|$, denote $$ M_i := \bigoplus_{|I| = i} pind_I^- \circ pres_I [-d_I].$$ Then the functor $\DG$ is glued\footnote{see definition \ref{def glue}.} from the list of functors $ M_0 , \ldots , M_{|\Sigma|}$.
\end{theorem}

\begin{proof}
	The proof is postponed to section \ref{sec prf of thm main}.
\end{proof}

\begin{corollary}\label{clm DG id on cusp}
	The functor $\DGj_{D (G \backslash G)}$ is isomorphic to $\textnormal{Id}_{D(G \backslash G)} [-d_{\Sigma}]$ when restricted to $D(G \backslash G)^{cusp}$ (the subcategory consisting of objects $\calF$ for which $pres_I (\calF) = 0$ for all $I \neq \Sigma$).
\end{corollary}

\begin{corollary}\label{cor DG preserves CH}
	The functor $\DG$ preserves the subcategory $\CH \subset D(G \backslash G)$.
\end{corollary}

\begin{proof}
	Clear by theorem \ref{thm main} and lemma \ref{lem pres pind preserve ss}.
\end{proof}

\subsection{Invertibility}

\begin{proposition}\label{DG invertible}
	The functor $\DG$ is invertible.
\end{proposition}

\begin{proof}
	By corollary \ref{cor DG inv if prp}, it is enough to show that $\DG$ admits a proper right adjoint. This follows from theorem \ref{thm main}, since by corollary \ref{cor second adj} $pind^-_I$ and $pres_I$ admit all iterated adjoints.
	\quash{From theorem \ref{thm cusp decomp} it follows that $D(G \backslash G)$ is generated by compact objects of the form $pind_I \calG$ where $\calG$ is compact and cuspidal. Thus it is enough to check that $\DG (pind_I \calG)$ is compact. But $\DG (pind_I \calG) \cong pind^-_I (\DGj_{D(G_I / G_I)} \calG) \cong pind^-_I (\calG [-d_I])$ so that the compactness follows.}
\end{proof}

\begin{proposition}\label{clm formula for inv}
	For $0 \leq i \leq |\Sigma|$, denote $$ N_i := \bigoplus_{|I| = i} pind_I \circ pres_I [d_I].$$ Then the functor $\DGinv$ is glued from the list of functors $ N_{|\Sigma |} , \ldots , N_{0}$.
\end{proposition}

\begin{proof}
	Apply left adjoints to the gluing of theorem \ref{thm main}.
\end{proof}

\subsection{Deligne-Lusztig involution}

\begin{proposition}\label{prp DG strat-exact}
	Let $0 \leq i \leq |\Sigma |$.
	
	\begin{enumerate}
		\item The functor $\DG$ preserves $\CH^{(i)}$.
		\item The functor $\DG [d_i]$ is $t$-exact when restricted to $\CH^{(i)}$.
	\end{enumerate} 
\end{proposition}

\begin{proof}\
	\begin{enumerate}
		\item This follows from proposition \ref{clm DG commutes pind}.
		\item It is enough to show that if $\calF \in (\CH^{(i)})^{\ge 0}$ then $\DG (\calF) \in \CH^{\ge d_i}$ and $\DGinv (\calF) \in \CH^{\ge -d_i}$. Since the $t$-structure is compatible with filtered colimits, we can reduce to $\calF$ being of the form $\tau_{\ge 0} \calF^{\prime}$ where $\calF^{\prime}\in \CH^{(i)}$ is compact, and hence in particular to $\calF$ being of finite length. Thus we may reduce to $\calF$ being irreducible. By lemma \ref{lem irr uch cusp rank}, we may reduce to $\calF$ being of the form $pind_I \calG$ for some $|I| = i$ and cuspidal irreducible $\calG$. But then: $$\DG (pind_I \calG) \cong pind^-_I (\DGI (\calG))  \cong pind^-_I (\calG [-d_i]) \in \CH^{\ge d_i} $$ and similarly $$\DGinv (pind_I \calG) \cong pind^-_I (\DGIinv (\calG))  \cong pind^-_I (\calG [d_i]) \in \CH^{\ge -d_i}.$$
	\end{enumerate}
\end{proof}

\quash{\begin{corollary}\label{cor DG preserves fl}
	The functor $\DG$ preserves the property of being of finite $t$-length.
\end{corollary}}

\begin{corollary}
	We obtain an auto-bijection (the \emph{Deligne-Lusztig involution}) $$ \textnormal{DL} \lefttorightarrow \textnormal{Irr} \left( \CH^{\heartsuit} \right),$$ by sending an irreducible object $\calE \in (\CH^{(i)})^{\heartsuit}$ to $\DGj_{D( G \backslash G)}(\calE) [d_i]$.
\end{corollary}

Let us show that $\textnormal{DL}$ is indeed involutive:

\begin{claim}\label{DL involutive}
	One has $\textnormal{DL} \circ \textnormal{DL} = \textnormal{Id}$.
\end{claim}

\begin{proof}
	In view of theorem \ref{thm main} and proposition \ref{clm formula for inv}, coupled with lemma \ref{lem pind pind- same}, it is clear that $\DG$ and $\DGinv$ induce the same map $$ K_0 (\CH^{fl}) \to K_0 (\CH^{fl}).$$
\end{proof}

\begin{remark}
	From theorem \ref{thm main}, it is clear that the involution $\textnormal{DL}$
	coincides with that defined in \cite[\S 15]{Lu1} (where it is denoted ``$d$").
\end{remark}

\section{The Deligne-Lusztig involution applied to unipotent principal series}\label{sec DL on prin ser}

In this section we record the effect of the Deligne-Lusztig involution when applied to unipotent principal series character $D$-modules, as an illustration. This is just a special case of \cite[Corollary 15.8.(c)]{Lu1}, but we include it here for completeness.

\medskip

In this section we will work with abelian categories; all $\textnormal{Hom}$'s will be understood to be $H^0 \textnormal{Hom}$'s, etc.

\subsection{Statement}

Let us denote $\textnormal{Spr} := pind_{\emptyset} (\calO)$ (this is the \emph{Springer $D$-module}). The irreducible character $D$-modules which are isomorphic to a direct summand of $\textnormal{Spr}$ will be called \emph{unipotent principal series}. The set of isomorphism classes of unipotent principal series character $D$-modules is in bijection with the set $Irr(W)$ of isomorphism classes of irreducible representations of $W$ (see section \ref{sec PS}); Denote by $\textnormal{Spr}_{\alpha}$ the isomorphism class of unipotent principal series character $D$-modules corresponding to $\alpha \in Irr(W)$.

\begin{proposition}\label{prp DL on Spr alpha}
	For every $\alpha \in \textnormal{Irr} (W)$, one has $$\textnormal{DL} (\textnormal{Spr}_{\alpha}) = \textnormal{Spr}_{\textnormal{sgn} \otimes \alpha}$$ (where $\textnormal{sgn}$ is the one-dimensional sign representation of $W$).
\end{proposition}

\begin{proof}
	The proof (or, rather, a reduction to a calculation for Weyl groups, for which references are given) is postponed to the end of section \ref{sec DL Weyl}.
\end{proof}

\subsection{The abelian category $\textnormal{PS} (G)$ and its relation to $\textnormal{Rep}_k^{fd} (W)$}\label{sec PS}

Let us denote by $\textnormal{PS}(G) \subset \CH^{\heartsuit}$ the abelian subcategory consisting of objects which are isomorphic to finite direct sums of direct summands of the Springer $D$-module $\textnormal{Spr} := pind_{\emptyset} (\calO)$. Notice that $\textnormal{PS} (G) \subset (\CH^{(0)})^{\heartsuit}$.

\begin{lemma}
	The functors $pind_I$, $pres_I$ and $\textnormal{DG}_{D(G \backslash G)}[d_{\emptyset}]$ preserve $\textnormal{PS} (\cdot)$.	
\end{lemma}

\begin{proof}
	For $pind_I$ and $pres_I$ the claim is clear by remark \ref{rem properties of pind} and lemma \ref{lem pres pind preserve ss}. For $\textnormal{DG}_{D(G \backslash G)}[d_{\emptyset}]$ the claim is clear because $\textnormal{DG}_{D(G \backslash G)} [d_{\emptyset}] (\textnormal{Spr}) = \textnormal{Spr}$ in view of proposition \ref{clm DG commutes pind}.
\end{proof}

Recall the following (Springer theory): The restriction of the $D$-module $\textnormal{Spr}$ to $$N_G (T) \backslash T^{reg} \cong G \backslash G^{rs} \subset G \backslash G,$$ under the further identification $$ D(W \backslash T^{reg})^{\heartsuit} \cong D(N_G (T) \backslash T^{reg})^{\heartsuit},$$ becomes identified with the pushforward of $\calO$ via $T^{reg} \to W \backslash T^{reg}$. By permuting the fibers of this Galois cover, one obtains an action of $W$ on $\textnormal{Spr}$, which induces an isomorphism $\textnormal{End} (\textnormal{Spr}) \cong k[W]$. One therefore obtains an equivalence of abelian categories $$ F : \textnormal{PS} (G) \xrightarrow{\approx} \textnormal{Rep}_k^{fd} (W),$$ given by $\textnormal{Hom} (\textnormal{Spr} , -)$. In particular, the isomorphism classes of irreducible summands of $\textnormal{Spr}$ are in bijection with the isomorphism classes of irreducible representations of $W$, and for $\alpha$ being one of the latter we denote by $\textnormal{Spr}_{\alpha}$ the corresponding one of the former.

\begin{claim}\label{clm F comm with pind pres}
	The following diagrams are commutative:
	
	$$ \xymatrix@R+1pc@C+2pc{ \textnormal{PS} (G) \ar[r]^{F} \ar[d]_{pres_I} & \textnormal{Rep}^{fd}_k (W) \ar[d]^{res^W_{W_I}} \\ \textnormal{PS} (G_I) \ar[r]_{F_I} & \textnormal{Rep}^{fd}_k (W_I) },$$
	
	$$ \xymatrix@R+1pc@C+2pc{ \textnormal{PS} (G) \ar[r]^{F} & \textnormal{Rep}^{fd}_k (W) \\ \textnormal{PS} (G_I) \ar[r]_{F_I} \ar[u]^{pind_I} & \textnormal{Rep}^{fd}_k (W_I) \ar[u]_{ind^W_{W_I}}}.$$
	
\end{claim}

\begin{proof}
	
	One has $\textnormal{Spr} \cong pind_I (\textnormal{Spr}_I)$. This identification can be chosen so that the two actions of $W_I$ on $\textnormal{Spr}$, one obtained by restriction from the action of $W$, and the other obtained by functoriality of $pind_I$ applied to the $W_I$ action on $\textnormal{Spr}_I$, coincide.
	
	Then $$ res^W_{W_I} (F (-)) = res^W_{W_I} \textnormal{Hom}(\textnormal{Spr} , -) \cong \textnormal{Hom}(pind_I (\textnormal{Spr}_I) , -) \cong \textnormal{Hom}(\textnormal{Spr}_I , pres_I (-)) = F_I (pres_I (-)).$$
	
	\medskip
	
	The commutativity of the second diagram follows by taking left adjoints in the first diagram.

\end{proof}

\subsection{The Deligne-Lusztig operator for the Weyl group}\label{sec DL Weyl}

\begin{definition}
	Define the \emph{Deligne-Lusztig operator} $$\textnormal{DL}_W : K_0 (\textnormal{Rep}^{fd}_k(W)) \to K_0 (\textnormal{Rep}^{fd}_k(W))$$ as follows: $$ \textnormal{DL}_W (V) := \sum_{I \subset \Sigma} (-1)^{|I|} \cdot ind^W_{W_I} res^W_{W_I} V.$$
\end{definition}

\begin{claim}\label{clm DL W}
	One has $$ \textnormal{DL}_W = \textnormal{sgn} \otimes -,$$ where $\textnormal{sgn}$ is the one-dimensional sign representation.
\end{claim}

\begin{proof}
	For the proof, see \cite[Theorem 2]{So} or \cite[Theorem 1]{Ka}.
\end{proof}

We will denote by $$\textnormal{DL} : K_0 (\textnormal{PS} (G)) \xrightarrow{\sim} K_0 (\textnormal{PS} (G))$$ the isomorphism given by $\textnormal{DL}$ on irreducibles (in other words, the isomorphism induced by $\textnormal{DG}_{D(G \backslash G)}[d_{\emptyset}]$).

\begin{lemma}\label{lem DL and DG match on Kgrp}
	One has, on the level of $K_0$-groups: $$ \textnormal{DL}_W \circ F = F \circ \textnormal{DL}.$$
\end{lemma}

\begin{proof}
	This is clear by theorem \ref{thm main} and claim \ref{clm F comm with pind pres}.
\end{proof}

\begin{proof}[Proof (of proposition \ref{prp DL on Spr alpha}).]
	One has $$ F (\textnormal{DL}  (\textnormal{Spr}_{\alpha})) = \textnormal{DL}_W (F(\textnormal{Spr}_{\alpha})) = \textnormal{DL}_W (\alpha) = \textnormal{sgn} \otimes \alpha = F (\textnormal{Spr}_{\textnormal{sgn} \otimes \alpha})$$ (where the first equality is by lemma \ref{lem DL and DG match on Kgrp} and the third equality is by claim \ref{clm DL W}) so $$ \textnormal{DL} (\textnormal{Spr}_{\alpha}) = \textnormal{Spr}_{\textnormal{sgn} \otimes \alpha}.$$
\end{proof}

\section{Proof of theorem \ref{thm main}}\label{sec prf of thm main}

In this section we provide a proof for theorem \ref{thm main}.

\subsection{The Vinberg monoid}

Our reference is \cite[Appendix D]{DrGa2}.

\medskip

Denote $T_{adj} := T / Z(G)$. We have an isomorphism $T_{adj} \xrightarrow{\sim} \bbG_m^{\Sigma}$ given by the simple roots. We denote by $\bar{\bbG}_m$ the affine line with its multiplicative monoid structure, and $\bar{T}_{adj} := \bar{\bbG}_m^{\Sigma}$.

\medskip

Let us denote by $V$ the Vinberg monoid of $G$. The group of invertible elements $V^{\times} \subset V$ is $G \overset{Z(G)}{\times} T$. One has a homomorphic map $deg : V \to \bar{T}_{adj}$ such that $deg^{-1} (T_{adj}) = V^{\times}$. Let us denote by $\overset{\circ}{V} \subset V$ the non-degenerate locus (it contains $V^{\times}$). The restriction of $deg$ to $\overset{\circ}{V} \subset V$ is smooth.

\medskip

We consider $V$ as a $(G \times G)$-space on the left, and a $T$-space on the right, by $(g_1 , g_2) * v * t = g_1 v g_2^{-1} t$. The right action of $T$ on $\overset{\circ}{V}$ is free, and the quotient by this action, as a $(G \times G)$-space, is the wonderful compactification of $G_{adj} := G / Z(G)$.

\medskip

Given $I \subset \Sigma$, let us denote by $(\bar{T}_{adj})_I \subset \bar{T}_{adj}$ the subspace of elements whose $(\Sigma - I)$-coordinates are $0$. Let us denote by $e_I \in (\bar{T}_{adj})_I$ the element whose $I$-coordinates are $1$. Denote by $T^I \subset T$ the subgroup consisting of elements whose $I$-coordinates are $1$.

\medskip

Recall our fixed choice of a Torel $T_{sub} \subset B \subset G$, giving rise to an identification $\phi : T \xrightarrow{\sim} T_{sub}$. Using this choice, we obtain a homomorphic section $s : \bar{T}_{adj} \to \overset{\circ}{V}$ of $deg$, which sends $t \in T_{adj}$ to $(\phi(t)^{-1} , t) \in V^{\times}$. Given $I \subset \Sigma$, the action of $G \times G$ on the fiber $\overset{\circ}{V}_{e_I}$ is transitive, and the stabilizer of $s(e_I)$ in $G \times G \times T^{op}$ consists of triples $(p_1 , p_2 , t)$ for which $p_1 \in P , p_2 \in P^- , t \in T^I$ and $[p_1] \cdot \phi (t) = [p_2]$. In particular, the stabilizer of $s(e_I)$ in $G \times G$ is $P \underset{M}{\times} P^-$.

\subsection{Filtration of the kernel}\label{sec filt of ker}

Let us denote by $S \subset G \times G \times \overset{\circ}{V}$ the subgroup scheme of the constant group scheme over $\overset{\circ}{V}$, consisting of $(g_1 , g_2 , m)$ for which $g_1 m g_2^{-1} = m$ (see \cite[subsection D.4.6]{DrGa2}, where $S$ is denoted $\textnormal{Stab}_{G \times G}$). Notice that $G \times G$ acts on $S$ on the left and $T$ acts on $S$ on the right (compatibly with these actions on $\overset{\circ}{V}$).

\medskip

We consider the following diagram, with Cartesian squares:

$$ \xymatrix{
	& & (G \times G) \ \backslash \ G \times G & & \\
	(G \times G) \ \backslash \ S_{e_{\Sigma}} \ar[r]^{\wt{\tau}} \ar[d]^{\pi^{\prime \prime}} & (G \times G) \ \backslash \ S_{\Sigma} \ / \ T \ar[d]^{\pi_{\Sigma}} \ar[r]^{\wt{j}} & (G \times G) \ \backslash \ S \ / \ T \ar[u]_{pr_{1,2}} \ar[d]^{\pi} & (G \times G) \ \backslash \ S_{I} \ / \ T \ar[l]_{\wt{i}} \ar[d]^{\pi_I} & (G \times G) \ \backslash \ S_{e_I} \ar[l]_{\wt{\sigma}} \ar[d]^{\pi^{\prime}} \\
	\{ e_{\Sigma} \} \ar[r]^{\tau} & (\bar{T}_{adj})_{\Sigma} \ / \ T \ar[r]^{j} & \bar{T}_{adj} \ / \ T & (\bar{T}_{adj})_{I} \ / \ T \ar[l]_{i} & \{ e_I \}  \ar[l]_{\sigma} }
$$

Here, $\pi$ is the projection induced by $S \xrightarrow{pr_3} \overset{\circ}{V} \xrightarrow{deg} \bar{T}_{adj}$. Notice that $\pi$ is smooth because $pr_3$ is (see \cite[Section D.4.6]{DrGa2}) and $deg$ is (see \cite[Section D.4.5]{DrGa2}).

\medskip

The map $pr_{1,2}$ is projective - this follows from $\overset{\circ}{V} / T$ being projective.

\medskip

Let us describe explicitly the $(G \times G)$-space $S_{e_I}$. It can be identified with the subspace of $$ G \times G \times \left( (G \times G) / (P_I \underset{G_I}{\times} P_I^-) \right)$$ consisting of $(g_1 , g_2 ; x_1 , x_2)$ for which $(x_1^{-1} g_1 x_1 , x_2^{-1} g_2 x_2) \in P_I \underset{G_I}{\times} P_I^-$. The identification is obtained by sending $(g_1 , g_2 ; x_1 , x_2)$ to $(g_1 , g_2 , x_1 s(e_I) x_2^{-1})$.

\medskip

In particular, one sees that $(G \times G) \backslash S_{e_{\Sigma}}$ can be identified with $G \backslash G$, in such a way that $pr_{1,2} \circ \wt{j} \circ \wt{\tau}$ becomes identified with the diagonal embedding for $G \backslash G$.

\medskip

The Cousin complex now allows us to glue the kernel representing $\DG$, which is $(pr_{1,2})_* \wt{j}_! \wt{\tau}_! C$, from kernels of the form $(pr_{1,2} \circ \wt{i})_* \wt{i}^! \wt{j}_! \wt{\tau}_! C$. We will thus prove theorem \ref{thm main} if we show:

\begin{claim}\label{clm calc of filt}
	The kernel $$(pr_{1,2} \circ \wt{i})_* \wt{i}^! \wt{j}_! \wt{\tau}_! C \in D(G \backslash G \times G \backslash G)$$ corresponds to the functor $pind^- \circ pres_I [-d_I]$.
\end{claim}

\subsection{Calculation of the filtrants}

\begin{lemma}\label{lem calc 1}
	One has $ i^! j_! \tau_! C \cong \sigma_* \omega [-d_I]$.
\end{lemma}

\begin{proof}
	By the contraction principle (see \cite[Proposition 3.2.2]{DrGa3}), denoting by $r : \bar{T}_{adj} / T \to (\bar{T}_{adj})_I / T $ the map equating all $(\Sigma - I)$-coordinates to $0$, one has $$ i^! j_! \tau_! C \cong r_! j_! \tau_! C = (r \circ j \circ \tau)_! C = \sigma_! C = \sigma_! \omega.$$ Notice that one has $\sigma_! = \sigma_* [-d_I]$ (for example, this follows from claim \ref{clm DG commutes} and the calculation of Pseudo-Identity in \cite[Section 6.7.3]{Ga1}) and thus the claim follows.
\end{proof}

\begin{lemma}\label{lem calc 2}
	One has $ \wt{i}^! \wt{j}_! \wt{\tau}_! C \cong \wt{\sigma}_* \omega [-d_I]$.
\end{lemma}

\begin{proof}
	We have $$\wt{i}^! \wt{j}_! \wt{\tau}_! C = \pi_I^* i^! j_! \tau_! C \cong \pi_I^* \sigma_* \omega [-d_I] = \wt{\sigma}_* (\pi^{\prime})^* \omega [-d_I] = \wt{\sigma}_* \omega [-d_I].$$
\end{proof}

\begin{lemma}\label{lem calc 3}
	One has $ (pr_{1,2} \circ \wt{i})_* \wt{i}^! \wt{j}_! \wt{\tau}_! C \cong (pr_{1,2} \circ \wt{i} \circ \wt{\sigma})_* \omega [-d_I]$.
\end{lemma}

\begin{proof}
	Follows from \ref{lem calc 2}.
\end{proof}

\subsection{Description by correspondence}

From lemma \ref{lem calc 3}, we see that the kernel $$(pr_{1,2} \circ \wt{i})_* \wt{i}^! \wt{j}_! \wt{\tau}_! C \in D(G \backslash G \times G \backslash G)$$ corresponds to the shift by $-d_I$ of  $(pr_1)_* pr_2^!$ where $$ \xymatrix{& (G \times G) \ \backslash \ S_{e_I} \ar[ld]_{pr_1} \ar[rd]^{pr_2} & \\ G \ \backslash \ G & & G \ \backslash \ G}.$$ Using the concrete description of $S_{e_I}$ in section \ref{sec filt of ker}, it is easy to identify this correspondence with the correspondence defining $pind_I^- \circ pres_I$, thus settling claim \ref{clm calc of filt}, and with it theorem \ref{thm main}.

\end{document}